\documentclass[12pt]{amsart}
\usepackage{fullpage}
\usepackage{times}
\usepackage{latexsym}
\usepackage{amssymb}
\usepackage{pinlabel}
\usepackage[all]{xy}

\usepackage{xcolor}

\newtheorem{thm}{Theorem}[section]

\newtheorem{cor}[thm]{Corollary}

\newtheorem{lem}[thm]{Lemma}
\newtheorem{que}[thm]{Question}

\theoremstyle{remark}
\newtheorem{rem}[thm]{Remark}

\usepackage{hyperref}

\newcommand{\Q}{\mathbb{Q}}
\newcommand{\R}{\mathbb{R}}
\newcommand{\Z}{\mathbb{Z}}

\newcommand{\SO}{\mathrm{\SO}}
\newcommand{\SL}{\mathrm{SL}}

\newcommand{\hooklongrightarrow}{\lhook\joinrel\longrightarrow}

\title[The simplicial volume of mapping tori of 3-manifolds]{The simplicial volume of mapping tori of 3-manifolds}

\author{Michelle Bucher and Christoforos Neofytidis}
\address{Section de Math\'ematiques, Universit\'e de Gen\`eve, 2-4 rue du Li\`evre, Case postale 64, 1211 Gen\`eve 4, Switzerland}
\email{Michelle.Bucher-Karlsson@unige.ch}
\address{Section de Math\'ematiques, Universit\'e de Gen\`eve, 2-4 rue du Li\`evre, Case postale 64, 1211 Gen\`eve 4, Switzerland}
\email{Christoforos.Neofytidis@unige.ch}
\date{\today}
\subjclass[2010]{57M05, 55R10, 57R22, 53C23}
\keywords{Simplicial volume, mapping torus, fiber bundle, mapping class group of 3-manifold}

\begin{document}

\maketitle

\begin{abstract}
We prove that any mapping torus of a closed 3-manifold has zero simplicial volume. When the fiber is a prime 3-manifold, classification results can be applied to show vanishing of the simplicial volume, however the case of reducible fibers is by far more subtle. We thus analyse the possible self-homeomorphisms of reducible 3-manifolds, and use this analysis to produce an explicit representative of the fundamental class of the corresponding mapping tori. To this end, we introduce a new technique for understanding self-homeomorphisms of connected sums in arbitrary dimensions on the level of classifying spaces and for computing the simplicial volume. In particular, we extend our computations to mapping tori of certain connected sums in higher dimensions.
Our main result completes the picture for the vanishing of the simplicial volume of fiber bundles in dimension four. Moreover, we deduce that dimension four together with the trivial case of dimension two are the only dimensions where all mapping tori have vanishing simplicial volume. As a group theoretic consequence, we derive an alternative proof of the fact that the fundamental group $G$ of a mapping torus of a 3-manifold $M$ is Gromov hyperbolic if and only if $M$ is virtually a connected sum $\# S^2\times S^1$ and $G$ does not contain $\Z^2$. 
\end{abstract}

\section{Introduction}

For a topological space $X$ and a homology class $\alpha\in H_n(X;\R)$, Gromov~\cite{Gromov} introduced 
the $\ell^1$-semi-norm of $\alpha$ to be
\[
\|\alpha\|_1:=\inf \biggl\{\sum_{j} |\lambda_j| \ \biggl\vert  \ \sum_j \lambda_j\sigma_j\in C_n(X;\R) \ \text{is a singular cycle representing } \alpha  \biggl\}.
\]
If $M$ is a closed oriented $n$-dimensional manifold, then the simplicial volume of $M$ is given by $\|M\|:=\|[M]\|_1$, where $[M]$ denotes the fundamental class of $M$. 

In general, it is a hard problem to compute or estimate the simplicial volume (or more generally the $\ell^1$-semi-norm), and results have been obtained only for certain classes of spaces. Most notably, there are precise computations for hyperbolic manifolds \cite{Gromov, Th1}  and $\mathbb{H}^2\times\mathbb{H}^2$-manifolds~\cite{Buch1}, non-vanishing results for negatively curved manifolds~\cite{Gromov,IY} and higher rank locally symmetric spaces of non-compact type~\cite{Buch,LS,Savage}, and estimates for direct products~\cite{Gromov} and surface bundles~\cite{Buch2,HK}. The case of general fiber bundles seems to be more subtle. For instance, if the fiber of a bundle has amenable fundamental group, then the simplicial volume vanishes~\cite{Gromov}, however, if the bundle fibers over a base with amenable fundamental group, then the simplicial volume need not be zero. Indeed, any surface bundle over the circle that admits a hyperbolic structure has positive simplicial volume.  In contrast, surface bundles over manifolds of dimension greater than one and amenable fundamental group have zero simplicial volume~\cite{Bow}. It is therefore natural to ask what happens to the simplicial volume of higher dimensional bundles over spaces with amenable fundamental groups:

\begin{que}
Let $M$ be a closed oriented manifold of dimension greater than two. When does the simplicial volume of an $M$-bundle over a closed oriented manifold with amenable fundamental group vanish?
\end{que}

Our main result gives a complete answer in dimension four:

\begin{thm}\label{t:3-mfdoverS^1}
The simplicial volume of every mapping torus of a closed 3-manifold is zero.
\end{thm}

This result has various topological and group theoretic consequences. First of all, Theorem \ref{t:3-mfdoverS^1}, together with known calculations and estimates completes the picture for the vanishing of the simplicial volume of fiber bundles in dimension four:

\begin{cor}\label{c:fiber4}
Let $F\to M\to B$ be a closed 4-manifold which is a fiber bundle with fiber $F$ and base $B$. If $\dim(F)=\dim(B)=2$ and the surfaces $F$ and $B$ support a hyperbolic structure, then $\|M\|>0$; otherwise $\|M\|=0$.
\end{cor}

In higher dimensions, it is easy to construct examples of manifolds that fiber over the circle and have positive simplicial volume, by taking products of hyperbolic 3-manifolds that fiber over the circle with hyperbolic manifolds of any dimension. Since in addition the only mapping torus in dimension two is the 2-torus, which has zero simplicial volume, we obtain the following:

\begin{cor}\label{c:mappingtorizero}
The simplicial volume of every mapping torus vanishes only in dimensions two and four.
\end{cor}

An interesting group theoretic problem is to determine whether a given group is Gromov hyperbolic. A well-known obstruction for a group to be hyperbolic is to contain a subgroup isomorphic to $\Z^2$. This property characterizes hyperbolicity for free-by-cyclic groups: If $F_m$ is a free group, then a semi-direct product $G=F_m\rtimes\Z$ is hyperbolic if and only if $G$ does not contain $\Z^2$ \cite{BF1,Br}. This in particular determines when the fundamental group of a mapping torus of a connected sum $\#_m S^2\times S^1$ is hyperbolic. Combining Theorem \ref{t:3-mfdoverS^1} with Mineyev's work on bounded cohomology of hyperbolic groups~\cite{Min1,Min2}, we obtain a topological proof of the following known characterization of 3-manifold-by-cyclic groups:

\begin{cor}\label{c:hyperbolic}
The fundamental group $G$ of a mapping torus of a 3-manifold $M$ is hyperbolic if and only if $M$ is virtually $\#_m S^2\times S^1$ and $G$ does not contain $\Z^2$.
\end{cor}

\begin{rem}\label{r:K-S}
As mentioned above, the characterization of the hyperbolicity of free-by-cyclic groups is part of~\cite{BF1,Br}. Thus, in view of geometrization of 3-manifolds, the only challenging case is when $M$ is (virtually) a connected sum of hyperbolic 3-manifolds and copies of $S^2\times S^1$ (with at least one hyperbolic 3-manifold). In this case, $\pi_1(M)$ is hyperbolic. M. Kapovich and Z. Sela  made us aware of the fact that if $H$ is a finitely presented hyperbolic group and $H\rtimes_\psi\Z$ is hyperbolic, then $H$ is virtually a free product of free and surface groups. In particular the virtual cohomological dimension of $H$ is $\leq 2$ which excludes $H=\pi_1(M)$. Thus $\pi_1(M)\rtimes_{f_*}\Z$ cannot be hyperbolic for any self-homeomorphism $f\colon M\longrightarrow M$. Furthermore, M. Kapovich informed us that the assumption on the hyperbolicity for $H$ can be removed. Nevertheless, Corollary \ref{c:hyperbolic} gives another uniform argument for all 3-manifolds (possibly prime) with an aspherical summand in their prime decomposition.
\end{rem}

In order to prove Theorem \ref{t:3-mfdoverS^1}, we need to examine self-homeomorphisms of the 3-manifold fiber $M$. When $M$ is reducible, then the monodromy of the mapping torus of $M$ is in general more complicated than when $M$ is prime~\cite{Mc,Zhao,NW}. However, as we shall see, the induced automorphism on the fundamental group of a reducible 3-manifold $M$ has a specific form. More precisely, by the description of any self-homeomorphism $f$ of 
$$M=M_1\#\cdots\# M_n\#(\#_m S^2\times S^1)$$
given in~\cite{Mc,Zhao} (the $M_i$ can be taken to be aspherical up to finite covers), 
the induced automorphism $f_*$ on $\pi_1(M)=\pi_1(M_1)*\dots *\pi_1(M_n)*F_m$, where $F_m$ denotes the free group of rank $m$, is a finite composition of self-automorphisms of factors $\pi_1(M_i)$ and $\pi_1(S^2\times S^1)$, of interchanges of isomorphic factors, and of automorphisms of type
\begin{equation}\begin{array}{rcll}\label{equ: slide M}\tag{4a}
\pi_1(M_1)*\dots *\pi_1(M_n)*F_m&\longrightarrow & \pi_1(M_1)*\dots *\pi_1(M_n)*F_m&\\
 \gamma_i &\longmapsto & \delta \gamma_i \delta^{-1}, \hspace{8pt} \mathrm{for \ } \gamma_i\in \pi_1(M_i), &\\
  \gamma_j&\longmapsto &\gamma_j, \hspace{30pt} \mathrm{for \ } \gamma_j\in \pi_1(M_j), j\neq i, \mathrm{\ or \ }\gamma_j\in F_m&
\end{array}\end{equation}
and
\begin{equation}\begin{array}{rcll}\label{equ: slide S}\tag{4b}
\pi_1(M_1)*\dots *\pi_1(M_n)*F_m&\longrightarrow & \pi_1(M_1)*\dots *\pi_1(M_n)*F_m&\\
 \gamma_j &\longmapsto & \delta \gamma_j, \hspace{2pt} \mathrm{for \ } \gamma_j\in \pi_1(j\mathrm{-th \ } S^1 \times S^2), &\\
  \gamma_i&\longmapsto &\gamma_i, \hspace{9pt} \mathrm{for \ } \gamma_i\in \pi_1(M_i), \forall i, \mathrm{\ or \ }\gamma_i \in \pi_1(i\mathrm{-th \ } S^1 \times S^2), i\neq j&
\end{array}\end{equation}
or
\begin{equation*}\begin{array}{rcll}
\pi_1(M_1)*\dots *\pi_1(M_n)*F_m&\longrightarrow & \pi_1(M_1)*\dots *\pi_1(M_n)*F_m&\\
 \gamma_j &\longmapsto & \gamma_j \delta, \hspace{2pt} \mathrm{for \ } \gamma_j\in \pi_1(j\mathrm{-th \ } S^1 \times S^2), &\\
  \gamma_i&\longmapsto &\gamma_i, \hspace{9pt} \mathrm{for \ } \gamma_i\in \pi_1(M_i), \forall i, \mathrm{\ or \ }\gamma_i \in \pi_1(i\mathrm{-th \ } S^1 \times S^2), i\neq j,&
\end{array}\end{equation*}
where $\delta\in\pi_1(M)$. Automorphisms of type (\ref{equ: slide M}) and (\ref{equ: slide S}) correspond to slide homeomorphisms of $M_i$ and of each end of $S^2\times I$ respectively; see Section \ref{s:reducible}.

Using the above description of $f_*$, we can describe any homeomorphism $f\colon M\longrightarrow M$ explicitly on the classifying space $B\pi_1(M)$ and apply techniques from bounded cohomology to  reduce our computation to the simplicial volume of a mapping torus of each individual prime summand. Then the vanishing of the simplicial volume of $M$ follows from the vanishing of the simplicial volume of mapping tori of prime 3-manifolds. In fact, assuming the above description for $f_*$ and the vanishing of the simplicial volume of any mapping torus of each prime summand, we can prove the following more general statement in arbitrary dimensions:

\begin{thm}\label{t:general}
Let $M=M_1\#\cdots\#M_n\#(\#_m S^{q-1}\times S^1)$ be a closed oriented $q$-dimensional manifold, $q\geq3$, where $M_i$ are aspherical, and $f\colon M\longrightarrow M$ be a homeomorphism. Suppose that the induced automorphism $f_*\colon\pi_1(M)\longrightarrow\pi_1(M)$ is given by $(g_4g_3g_2g_1)_*$ where each $(g_\ell)_*$, $\ell=1,2,3,4$, is a finite composition of automorphisms of type ($\ell$) below:
\begin{itemize}
\item[(1)] self-automorphisms of factors $\pi_1(M_i)$;
\item[(2)] interchanges of two isomorphic factors;
\item[(3)] spins: self-automorphisms of $\pi_1(S^{q-1}\times S^1)\cong\mathbb{Z}$;
\item[(4)] slides of $M_i$ and each end of $S^{q-1}\times S^1$: automorphisms as given by (\ref{equ: slide M}) and (\ref{equ: slide S}) respectively.
\end{itemize}
If any mapping torus of each $M_i$ has zero simplicial volume, then $\|M\rtimes_f S^1\|=0$.
\end{thm}

\begin{rem}
The terminology ``spins" and ``slides" comes from $3$-dimensional topology; see~\cite{Mc} or Section \ref{s:reducible}. Although from our point of view, we would not need to distinguish between type (1) and type (3), we prefer to stay close to the description from~\cite{Mc}.
\end{rem}

The above statement applies in particular to hyperbolic manifolds of dimension greater than two:

\begin{cor}\label{c:connectedsumshyperbolic}
Let $M=M_1\#\cdots\#M_n$, where each $M_i$ is a closed oriented hyperbolic manifold of dimension greater than two. Then $\|M\rtimes_f S^1\|=0$ for any homeomorphism $f\colon M\longrightarrow M$.
\end{cor}

Invoking again Mineyev's work~\cite{Min1,Min2}, we can derive that for $M$ as in Corollary \ref{c:connectedsumshyperbolic}, the semi-direct product $\pi_1(M)\rtimes_{f_*}\Z$ is not hyperbolic. This conclusion holds more generally, for $M$ as in Theorem \ref{t:general}, provided that an aspherical summand $M_i$ exists in the prime decomposition of $M$. As pointed out in Remark \ref{r:K-S}, a purely group theoretic argument can be given to show that $\pi_1(M)\rtimes_{f_*}\Z$ is not hyperbolic.

\subsection*{Outline}
In Section \ref{s:preliminary} we collect some simple observations on mapping tori. In Section \ref{s:general} we prove Theorem \ref{t:general}. Section \ref{s:3-mfd} is devoted to the proof of Theorem \ref{t:3-mfdoverS^1}. In Section \ref{s:remarks} we discuss briefly Corollaries \ref{c:fiber4}, \ref{c:mappingtorizero}, \ref{c:hyperbolic} and \ref{c:connectedsumshyperbolic}.

\subsection*{Acknowledgements}
We would like to thank Brian Bowditch, Misha Kapovich, Jean-Fran\c cois Lafont, Clara L\"oh and Zlil Sela for useful comments. Both authors gratefully acknowledge support by the Swiss National Science Foundation, under grants FNS200021$\_$169685 and FNS200020$\_$178828.

\section{Preliminary observations}\label{s:preliminary}\label{subsubsection: fund class}

Let $X$ be a CW-complex and $f\colon X\longrightarrow X$ be a continuous (not necessarily bijective) map. Recall that the mapping torus $X \rtimes_f S^1$ of $f$, defined as 
$$ X\rtimes_f S^1=X\times [0,1]/((x,0)\sim (f(x),1),$$
naturally projects onto the circle, but it is not necessarily a fiber bundle when $f$ is not a homeomorphism. Its fundamental group is the semi-direct product
$$\pi_1(X\rtimes_f S^1)=\pi_1(X)\rtimes_{f_*}\mathbb{Z},$$
where the positive generator of $\mathbb{Z}$ acts on $\pi_1(X)$ by $\gamma\mapsto f_*(\gamma)$. There is a long exact sequence in integral homology
$$\dots \longrightarrow H_{q+1}(X)\overset{i_*}{\longrightarrow}H_{q+1}(X\rtimes_f S^1)\longrightarrow H_{q}(X)\overset{1-f_*}{\longrightarrow}H_{q}(X)\longrightarrow \dots,$$
where $i_*$ is induced by the inclusion $X=X\times \{0\}\hookrightarrow X\times [0,1]$. Furthermore, this long exact sequence is natural. More precisely, let $X'\rtimes_{f'} S^1$ be the mapping torus of $f'\colon X'\longrightarrow X'$, and suppose that there exists a continuous map $\varphi\colon X\longrightarrow X'$  such that the diagram
$$
\xymatrix{
X \ar[r]^{f} \ar[d]_\varphi & X \ar[d]^\varphi \\
X' \ar[r]^{f'}  &X' \\
}
$$
commutes up to homotopy. Then $\varphi$ induces a continuous map, still denoted $\varphi\colon X\rtimes_f S^1\longrightarrow X'\rtimes_{f'} S^1$, and the following diagram commutes:
\begin{equation}\label{les commute} 
\xymatrix{
\dots \ar[r] & H_{q+1}(X)\ar[r]^{i_* \ \  \ \ \ }\ar[d]^{\varphi_*}&H_{q+1}(X\rtimes_f S^1) \ar[d]^{\varphi_*}\ar[r]& H_{q}(X) \ar[d]^{\varphi_*}\ar[r]^{1-f_*}&H_{q}(X) \ar[d]^{\varphi_*}\ar[r]& \dots
\\ 
\dots \ar[r] & H_{q+1}(X')\ar[r]^{i_* \ \ \ \ \ }&H_{q+1}(X'\rtimes_{f'} S^1) \ar[r]& H_{q}(X') \ar[r]^{1-f'_*}&H_{q}(X') \ar[r]& \dots.
\\ }
\end{equation}

In our case, the CW-complex $X$ will always be $q$-dimensional, for $q\geq 3$, so that the above long exact sequence in homology becomes
\begin{equation}0\longrightarrow H_{q+1}(X\rtimes_f S^1)\longrightarrow H_{q}(X)\overset{1-f_*}{\longrightarrow}H_{q}(X)\longrightarrow \dots.\label{equ:les}\end{equation}
When $f$ is an orientation preserving self-homeomorphism of a closed oriented $q$-dimensional manifold $X$, we obtain an isomorphism $H_{q+1}(X\rtimes_f S^1)\cong H_q(X)\cong \mathbb{Z}$ which maps the fundamental class of the $(q+1)$-dimensional manifold $X\rtimes_f S^1$ to the fundamental class of $X$. We want to slightly generalize the notion of fundamental class to our setting:
\begin{itemize}
\item For $X=M_1\vee \dots \vee M_n\vee Y$, where $M_i$ are closed oriented $q$-dimensional manifolds and $Y$ is a $1$-dimensional CW-complex, we still call $[M_1]+\dots +[M_n]\in H_q(X)$ the fundamental class of $X$ and denote it by $[X]$.
\item If for $X$ as above, the continuous map $f\colon X\longrightarrow X$ satisfies
$$[X]=f_*([X])\in H_q(X),$$
then there is a unique class in $H_{q+1}(X\rtimes_f S^1)$ mapped to $[X]$ in the long exact sequence (\ref{equ:les}). We denote it by $[X\rtimes_f S^1]$ and call it the fundamental class of $X\rtimes_f S^1$. We further define the simplicial volume of $X\rtimes_f S^1$, denoted $\|X\rtimes_f S^1\|$, as the $\ell^1$-semi-norm of the fundamental class of $X\rtimes_f S^1$. 
\end{itemize}

To finish this preliminary section, let us quote the following general fact (see also~\cite{LN}), which will reduce our discussion to mapping tori of finite covers of the fiber.

\begin{lem}\label{l:finitecover}
Let $M$ be a manifold which has a finite cover $\overline{M}$ such that $\|\overline{M}\rtimes_g S^1\|=0$ for every homeomorphism $g\colon\overline{M}\longrightarrow\overline{M}$. Then $\|M\rtimes_f S^1\|=0$ for every homeomorphism $f\colon M\longrightarrow M$.
\end{lem}
\begin{proof}
By the multiplicativity of the simplicial volume under finite coverings, it suffices to show that any mapping torus $M\rtimes_f S^1$ is finitely covered by a mapping torus $\overline{M}\rtimes_g S^1$.

Suppose $f\colon M\longrightarrow M$ is a homeomorphism. Since $\pi_1(\overline{M})$ has finite index in $\pi_1(M)$ and $\pi_1(M)$ has finitely many subgroups of a fixed index (being finitely generated), there is some natural number $k$ such that $$f^k_*(\pi_1(\overline{M}))=\pi_1(\overline{M}).$$
The desired finite cover of $M\rtimes_f S^1$ is then given by the mapping torus $\overline{M}\rtimes_{f^k} S^1$. Indeed, $\overline{M}\rtimes_{f^k} S^1$ is a covering of degree $[\pi_1(M):\pi_1(\overline{M})]$ of the mapping torus $M\rtimes_{f^k} S^1$, which is a degree $k$ covering of $M\rtimes_f S^1$.
\end{proof}

\section{Proof of Theorem \ref{t:general}}\label{s:general}

As we mentioned in the introduction, the case of (reducible) 3-manifolds is contained in the more general statement of Theorem \ref{t:general}. We thus prove Theorem \ref{t:general} first.

\medskip

Let 
$$M = M_ 1 \# M_ 2  \# \cdots \# M_n \# (\#_m S ^1  \times S^{q-1})$$
be a closed oriented $q$-dimensional manifold, where $M_i$ are aspherical, and $f\colon M\longrightarrow M$ be a self-homeomorphism of $M$ such that the induced automorphism $f_*\colon\pi_1(M)\longrightarrow\pi_1(M)$ fulfils the assumptions of Theorem \ref{t:general}.

Set 
$$M^\vee:= M_1\vee \dots \vee M_n\vee (\vee_{j=1}^m S^1)$$
and observe that
$$\pi_1(M)=\pi_1(M^\vee)=\pi_1(M_1)*\cdots *\pi_1(M_n)* F_m.$$
Furthermore, $M^\vee$ is aspherical and hence a model for the classifying space $B\pi_1(M)$. Therefore, every self-homeomorphism of $M$ admits a counterpart on $M^\vee$, i.e. a self-map $M^\vee\longrightarrow M^\vee$ inducing the same map on the fundamental group. 

We will now give an explicit description of four types of maps inducing the four types of automorphisms that appear in Theorem \ref{t:general}. For this it is useful to think of $M$ as a punctured $q$-sphere $W$ with $n+2m$ open $q$-balls removed and where the summands $M_i$ and $S^1\times S^{q-1}$ are obtained as follows: Let
\begin{equation}\label{equ: spheres} 
S_1, S_2, ... , S_n, S_{n+1,0}, S_{n+1,1}, ..., S_{n+m,0}, S_{n+m,1}
\end{equation}
be the boundary components of $W$. For each summand $M_i$, $i=1,...,n$, choose a $q$-ball $D_i$ and attach $M_i'= M_i \setminus \mathrm{int}(D_i)$ to  $S_i$ along $\partial D_i$. For  $n+ 1\le j \le n+m$, let $S_j \times I$ be  a  copy of  $S^{q-1}\times I$  attached  to   $W$ by identifying $S_j \times   \{0\}$ with  $S_{j,0}$ and $S_{j} \times \{1\}$ with $S_{j,1}$ to form an $S^1  \times  S^{q-1}$ summand.
We can then explicitly describe (a model of) the classifying map
$$\varphi\colon M\longrightarrow M^\vee=B\pi_1(M).$$
Define $b\in M^\vee$ to be the join of the wedge product and assume that $b\in D_i$ for each $i=1,\dots,n$. The classifying map can be defined as follows:
\begin{itemize}
\item collapse $W$ to the join $b$ in $M^\vee$,
\item send $M_i'=M_i\setminus \mathrm{int}(D_i)$ homeomorphically (and canonically) to $M_i\setminus \{b\}$,
\item project the $j$-th $S^{q-1}\times I$ onto $I$ and further to the $j$-th circle $S^1$ of the bouquet of circles $\vee_{j=1}^m S^1$ in $M^\vee$. 
\end{itemize}
From this point of view, it is obvious how to define, for each automorphism of type 1, 2 or 3 the corresponding map (even a homeomorphism) on $M^\vee$ that not only induces the same map on $\pi_1(M)=\pi_1(M^\vee)$ but even commutes with the classifying map as given above. Type 4 automorphisms -- the slides -- require more care. Recall 
their description in (\ref{equ: slide M}) and (\ref{equ: slide S}) for $M_i$ and each end of $S^{q-1}\times I$ respectively.
Choose a representative $\alpha$ of $\delta\in \pi_1(M)=\pi_1(M^\vee)\cong \pi_1(M^\vee,b)$.

For slide automorphisms as in (\ref{equ: slide M}), define $g_4^\vee\colon M^\vee\longrightarrow M^\vee$ as follows: 
\begin{itemize}
\item it is the identity on $M_j$, for $j\neq i$, and on $\vee_{j=1}^m S^1$,
\item it maps $M'_i=M_i\setminus \mathrm{int}(D_i)$ homeomorphically (and canonically as we have chosen $b\in D_i$) to $M_i\setminus \{b\}$.
\item it is the composition
$$D_i\setminus \{b\}\cong S^{q-1}\times [0,1]\longrightarrow [0,1]\overset{\alpha}{\longrightarrow} M,$$
of the projection on the second factor with $\alpha$ on $D_i\setminus \{b\}$.
\end{itemize}
See Figure \ref{f:pinched} for an illustration of $g_4^\vee$. It is obvious by construction that $(g_4)_*=(g_4^\vee)_*$ on $\pi_1(M)=\pi_1(M^\vee)$.

\begin{figure}[ht!]
\labellist
\pinlabel $M_1$ at -15 25
\pinlabel $M_i$ at -5 275
\pinlabel $M_n$ at 305 255
\pinlabel $S^1$ at 240 25
\pinlabel $S^1$ at 310 125
\pinlabel $M_1$ at 415 25
\pinlabel $M_i$ at 425 275
\pinlabel $M_n$ at 735 255
\pinlabel $S^1$ at 670 25
\pinlabel $S^1$ at 740 125
\pinlabel ${\color{red} \alpha}$ at 590 55
\pinlabel $\stackrel{g_4^\vee}\longrightarrow$ at 375 150
\endlabellist
\centering
\includegraphics[width=13cm]{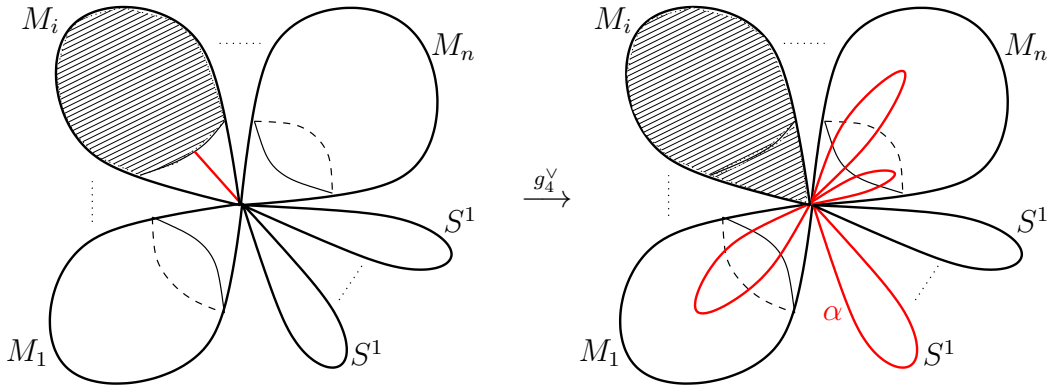}
\caption{\small The model space $M^\vee=M_1\vee\cdots\vee M_n\vee(\vee_{j=1}^m S^1)$ and the effect of $g_4^\vee$.}
\label{f:pinched}
\end{figure}

In the second case (i.e. for slide automorphisms as in (\ref{equ: slide S})), define $g_4^\vee$ as follows:
\begin{itemize}
\item it maps the $j$-th circle $S^1$ to the concatenation of $\alpha$ with the $j$-th circle, or of the $j$-th circle with $\alpha$,
\item it is the identity everywhere else.
\end{itemize}
Again it is obvious that  $(g_4)_*=(g_4^\vee)_*$ on $\pi_1(M)=\pi_1(M^\vee)$.

Set now
$$f^\vee\colon M^\vee\longrightarrow M^\vee$$
to be the map obtained 
by composing the self-maps on $M^\vee$ as described above, so that
$$f^\vee_*=(g_4^\vee)_*\circ (g_3^\vee)_* \circ (g_2^\vee)_*\circ (g_1^\vee)_*.$$

Since $f$ and $f^\vee$ induce identical maps on the fundamental group $\pi_1(M)=\pi_1(M^\vee)$ the diagram  
$$
\xymatrix{
M \ar[r]^{f} \ar[d]_\varphi & M \ar[d]^\varphi\\
M^\vee\ar[r]^{f^\vee}  &M^\vee \\
}
$$
commutes up to homotopy, and the classifying map $\varphi:M\longrightarrow M^\vee$ induces a map still denoted $\varphi$ between the corresponding mapping tori,
$$\varphi\colon M\rtimes_f S^1\longrightarrow M^\vee \rtimes_{f^\vee}S^1.$$
By construction, we have $\varphi_*([M])=[M^\vee]\in H_q(M^\vee)$ (where the fundamental class for not necessarily manifolds is as defined in Section \ref{subsubsection: fund class}). Thus the commuting long exact sequences in (\ref{les commute}) imply that $(1-f^\vee)_*([M^\vee])=0$, so that the fundamental class $[M\rtimes_f S^1]\in H_{q+1}(M\rtimes_f S^1)$ exists and
$$\varphi_*([M\rtimes_f S^1])=[M^\vee \rtimes_{f^\vee}S^1].$$
Since $\varphi$ induces an isomorphism between the fundamental groups of the two mapping tori, we obtain
$$\| M\rtimes_f S^1\| =\|M^\vee \rtimes_{f^\vee} S^1\|$$
as a consequence of Gromov's Mapping Theorem \cite[Section 3.1]{Gromov}. Our goal is now to prove $\|M^\vee \rtimes_{f^\vee} S^1\|=0$.

To this end, set
$$M^S:=M^\vee\vee (\vee_{k=1}^r S^1),$$
where $r$ is the number of slide automorphisms in the chosen decomposition of the original self-automorphism induced by $f$. Define a map 
$$A\colon M^S\longrightarrow M^\vee$$
as the identity on $M^\vee$ and sending the $k$-th circle $S^1$ of the bouquet of circles $\vee_{k=1}^r S^1$ to the loop $\alpha_k$ used to define the component $g_{4,k}^\vee$ of $g_4^\vee$. Let us now define a self-map $f^S\colon M^S\longrightarrow M^S$ for which the diagram
\begin{equation}\label{equ:MM commute}
\xymatrix{
M^\vee \ar[r]^{f^\vee}  & M^\vee \\
M^S \ar[u]^{A} \ar[r]^{f^S}  &M^S \ar[u]_{A} \\
}
\end{equation}
commutes. Recall that $f^\vee=g_4^\vee \circ g_3^\vee\circ g_2^\vee\circ g_1^\vee$. For $\ell=1,2,3$, define $g_\ell^S$ to be $g_\ell^\vee$ on $M^\vee$ and the identity on $\vee_{k=1}^r S^1$. Define $g_{4,k}^S$ as the identity on  $\vee_{k=1}^r S^1$ and almost precisely $g_{4,k}^\vee$ on $M^\vee$: If the slide automorphism $(g_{4,k})_*$ has the form (\ref{equ: slide M}), 
then $g_{4,k}^S$ is $g_{4,k}^\vee$ except for  $D_{i_{k}}\setminus \{b\}$, which is not mapped onto $\alpha_k$ but onto the $k$-th circle in the bouquet of circles $\vee_{k=1}^r S^1$ in $M^S$. While if the slide automorphism $(g_{4,k})_*$ has the form (\ref{equ: slide S}), 
then $g_{4,k}^S$ is $g_{4,k}^\vee$ except for the $j_k$-th circle $S^1$ of the first bouquet of circle $\vee_{j=1}^m S^1$ of $M^S$ which is not mapped to the concatenation of $\alpha$ and itself, but to the concatenation of the $k$-th circle in the (second) bouquet of circles $\vee_{k=1}^r S^1$ in $M^S$ and itself or vice versa. Each of these maps commutes with $A$, and hence so does the composition
$f^S=g_4^S \circ g_3^S\circ g_2^S\circ g_1^S,$
so that $A$ induces a map (still denoted by $A$) between the corresponding mapping tori
$$M^S\rtimes_{f^S} S^1 \overset{A}{\longrightarrow}M^\vee \rtimes_{f^\vee} S^1.$$
Clearly, $A_*([M^S])=[M^\vee]$ and by the commutativity of the diagram (\ref{equ:MM commute}) also
$$ A_*f_*^S([M^S])=f^\vee_*A_*([M^S])=f^\vee_*([M^\vee])=[M^\vee].$$
In view of the commuting long exact sequences from (\ref{les commute}) and the fact that $A_*\colon H_q(M^S)\longrightarrow H_q(M^\vee)$ is an isomorphism, we obtain that $(1-f^S_*)([M^S])=0$, so that the fundamental class $[M^S\rtimes_{f^S} S^1]\in H_{q+1}(M^S\rtimes_{f^S} S^1)$ is defined and satisfies $A_*([M^S\rtimes_{f^S} S^1])=[M^\vee \rtimes_{f^\vee}S^1]$. Therefore,
$$\| M^\vee \rtimes_{f^\vee}S^1\| \leq \| M^S\rtimes_{f^S} S^1\|.$$
It is thus sufficient to show that $\| M^S\rtimes_{f^S} S^1\|=0$.

\begin{figure}[ht!]
\labellist
\pinlabel $M_1$ at -5 30
\pinlabel $M_i$ at -10 275 
\pinlabel $M_n$ at 315 255
\pinlabel $\vee_{j=1}^mS^1$ at 175 -10
\pinlabel  $M_1$ at 425 30
\pinlabel $M_i$ at 419 275
\pinlabel $M_n$ at 745 255
\pinlabel $\vee_{k=1}^rS^1$ at 325 79
\pinlabel $\stackrel{g_4^S}\longrightarrow$ at 375 150
\endlabellist
\centering
\includegraphics[width=13cm]{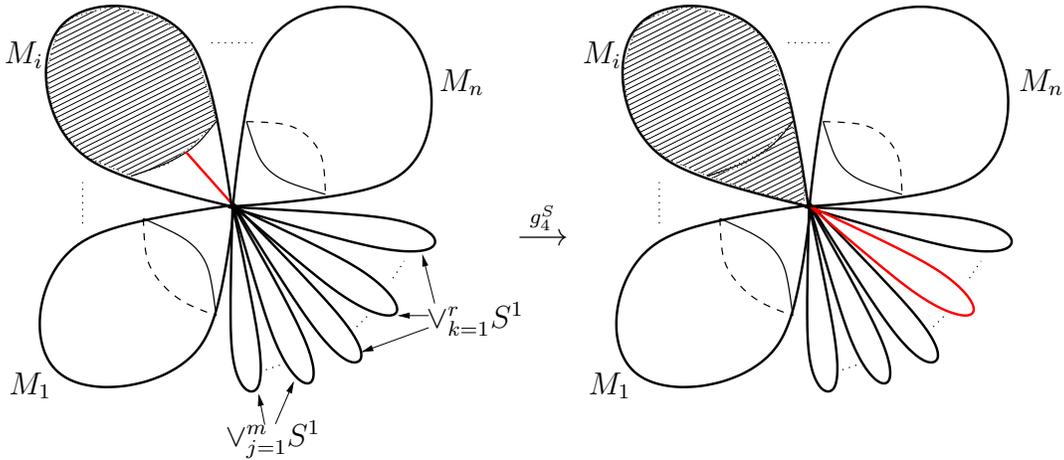}
\vspace{5pt}
\caption{\small The model space $M^\vee$ augmented by the bouquet $\vee_{k=1}^r S^1$ and the map $g_4^S$.}
\label{f:pinchedH}
\end{figure}

The advantage with this new mapping torus is that the slides have no more mixing effect on $M^S$. What we mean is that originally on $M^\vee$ (and also on $M$), the slide $g_{4,k}^\vee$ sends one irreducible summand, say $M_{j_k}$ to $M_{j_k}\cup \{\alpha_k\}$ (respectively $M_{j_k}$ union a neighborhood of $\alpha_k$), where $\alpha_k$ can lie all over $M$. Now, $g_{4,k}^S$ maps $M_{j_k} $ to the union of $M_{j_k}$ and the $k$-th circle of the (artificially) added wedge $\vee_{k=1}^r S^1$. Furthermore, $g_1^S$ and $g_3^S$ leave each $M_i$ invariant, while $g_2^S$ permutes them. Let $\mathcal{N}$ be the order of the permutation of $\{M_1,\dots,M_n,S^1,\dots,S^1\}$ induced by $g_3^S$. Then $(f^S)^\mathcal{N}$ maps each $M_i$ to $M_i\vee(\vee_{k=1}^r S^1)$ and is the identity on $\vee_{k=1}^r S^1$. Set $f_i:=(f^S)^\mathcal{N}\vert_{M_i\vee(\vee_{k=1}^r S^1)}$ and let $j_i\colon (M_i\vee (\vee_{k=1}^r S^1 ) ) \rtimes_{f_i} S^1\hooklongrightarrow M^S\rtimes_{(f^S)^\mathcal{N}} S^1$ denote the canonical inclusion. Then
\[
[M^S\rtimes_{(f^S)^\mathcal{N}} S^1]=\sum_{i=1}^n (j_i)_*[(M_i\vee (\vee_{k=1}^r S^1 ))\rtimes_{f_i} S^1].
\]
Since $M^S\rtimes_{(f^S)^\mathcal{N}} S^1$ is a finite cover of  $M^S\rtimes_{f^S} S^1$, we have
\[
\mathcal{N}\cdot \|M^S\rtimes_{f^S} S^1\|=\|M^S\rtimes_{(f^S)^\mathcal{N}}S^1\|\leq \sum_{i=1}^r \|(M_i \vee (\vee_{k=1}^r S^1))\rtimes_{f_i}S^1\|.
\]
It thus remains to show that $\|(M_i \vee( \vee_{k=1}^r S^1))\rtimes_{f_i}S^1\|=0$ for every $i$. 

By construction, the map $f_i\colon M_i\vee(\vee_{k=1}^r S^1)\longrightarrow M_i\vee(\vee_{k=1}^r S^1)$ is 
\begin{itemize}
\item a homeomorphism from $M_i\setminus \mathrm{int}(D_i)$ (for a possibly bigger ball $D_i$ than the one considered above) to $M_i\setminus \{b\}$,
\item the composition of the projection $D_i\setminus\{b\}\cong S^{q-1}\times I\longrightarrow I$ with a closed path in $\vee_{k=1}^r S^1$,
\item the identity on $\vee_{k=1}^r S^1$.
\end{itemize}

Define a map $j\colon M_i\longrightarrow M_i\vee (\vee_{k=1}^r S^1)$ as
\begin{itemize}
\item a (canonical) homeomorphism from $M_i\setminus \mathrm{int}(D_i)$ to $M_i\setminus \{b\}$,
\item $f_i$ on $D_i$. 
\end{itemize}
Observe that $j$ is homotopic to the canonical inclusion $i\colon M_i\hooklongrightarrow M_i\vee (\vee_{k=1}^r S^1) $. Define
$$F_i\colon M_i\longrightarrow M_i$$
as $j^{-1}\circ f_i$ on $M_i\setminus \mathrm{int}(D_i)$ and the identity on $D_i$ and observe that the diagram
$$
\xymatrix{
M_i \ar[r]^{F_i}  \ar[d]_{i} & M_i \ar[d]^{j} \\
M_i\vee (\vee_{k=1}^r S^1) \ar[r]^{f_i}  & M_i\vee  (\vee_{k=1}^r S^1)  \\
}
$$
commutes. Since $i$ and $j$ are homotopic they induce a map
$$I\colon M_i\rtimes_{F_i}S^1\longrightarrow (M_i\vee (\vee_{k=1}^r S^1) )\rtimes_{f_i}S^1.$$
As the map $i$ (or $j$) between the $q$-dimensional CW complexes induce an isomorphism on $H_q$ preserving the respective fundamental classes, the same holds for the induced map $I$ of the mapping tori (again using the commuting long exact sequences (\ref{les commute})). Thus in particular,

$$\| (M_i\vee(\vee_{k=1}^r S^1)) \rtimes_{f_i}S^1\|=\|I_*[M_i \rtimes_{F_i}S^1]\| \leq \| M_i \rtimes_{F_i}S^1\|=0 ,$$
where the last equality comes from the fact that  $F_i$ is a homeomorphism of $M_i$ and by assumption, the simplicial volume of any mapping torus of self-homeomorphisms of each prime summand $M_i$ vanishes.

This finishes the proof of Theorem \ref{t:general}.

\section{Proof of Theorem \ref{t:3-mfdoverS^1}}\label{s:3-mfd}

We now prove Theorem \ref{t:3-mfdoverS^1}. First, we use Theorem \ref{t:general} to reduce Theorem \ref{t:3-mfdoverS^1} to mapping tori of prime 3-manifolds.

\subsection{Mapping tori of reducible 3-manifolds}\label{s:reducible}
 
First, we recall the isotopy types of the orientation-preserving self-homeomorphisms of reducible 3-manifolds. For the discussion in this subsection, we follow McCullough's survey paper~\cite{Mc}, as  well as Zhao's paper~\cite{Zhao}. 

Suppose $M$ is a closed oriented reducible 3-manifold. By the Kneser-Milnor theorem, $M$ admits a non-trivial prime decomposition
$$M = M_ 1 \# M_ 2  \# \cdots \# M_n \# ( \#_m S ^1  \times S^2 ),$$    
where the summands  $M_i$ are irreducible and $m\ge 0$.  
As in the general case described in the beginning of Section \ref{s:general}, it is useful to view $M$ as a punctured $3$-cell $W$ with boundary components
\begin{equation}\label{equ: spheres3dim} S_1, S_2, ... , S_n, S_{n+1,0}, S_{n+1,1}, ..., S_{n+m,0}, S_{n+m,1},\end{equation}
and $M_i'= M_i \setminus \mathrm{int}(D_i)$ and $S^2\times I$  attached.  With this description, we can now list four types of homeomorphisms of $M$. 

\medskip

{\em Type 1. Homeomorphisms  preserving  summands.} These are the homeomorphisms that restrict to the identity on $W$. 

\medskip

{\em Type 2. Interchanges of homeomorphic summands.} If $M_i$ and  $M_j$ are orientation-preserving homeomorphic summands, then a homeomorphism of $M$ is given by fixing all other summands, leaving $W$ invariant, and interchanging $M'_i$ and $M'_j$. 

Similarly, we can interchange two $S^1 \times S^2$ summands, leaving $W$ invariant.

\medskip

{\em Type 3. Spins of $S^ 1 \times S^ 2$ summands.} For each $n + 1\le  j\le  n + m$, a homeomorphism of $M$ can be constructed by interchanging $S_{j,0}$ and $S_{j,1}$, restricting to an orientation-preserving homeomorphism that interchanges the boundary components of $S_j \times I$, and by fixing all other summands, leaving $W$ invariant.

\medskip

{\em Type 4. Slide homeomorphisms.} Let $S$ be one of the boundary spheres from (\ref{equ: spheres}), bounding either an $M'_i$, or an end of a copy of $I\times S^2$. Let $\alpha$ be an arc  in $M$ with start and endpoints in  $S$  intersecting $M_i$, respectively $S^1\times S^2$, only in its endpoints. Let $N\subset N'$ be two regular neighborhoods of $\alpha\cup S$ such that $N\subset \mathrm{Int}(N')$. Then $N'\setminus N$ has two connected components diffeomorphic to $S^2\times (0,1)$ and $T^2\times (0,1)$. Let $T$ denote the latter connected component. Its torus factor should be thought as the product $S^1_\alpha \times S^1$, where $S^1_\alpha$ is close to the path $\alpha$ closed up to a curve in $S$, and the second factor is radial. A slide homeomorphism $g_4\colon M\longrightarrow M$ is defined as the identity on $M\setminus T$, while on $T$ it is the product of a Dehn twist along $S^1_\alpha$ on $S^1_\alpha\times (0,1)$ and the trivial map on the radial factor. 

To understand the effect of $g_4$ on the fundamental group, observe that up to homotopy, any curve not intersecting $S$ is left unchanged by $g_4$ since it is possible to homotope it away from $\alpha\cup S$ and its regular neighborhood $N'$. In contrast, a curve entering through $S$ will be Dehn twisted along $S^1_\alpha$ by $g_4$. For the explicit description of $(g_4)_*\colon\pi_1(M)\longrightarrow\pi_1(M)$ we distinguish the case when $S$ bounds $M'_i$ or an end of a copy of $I\times S^2$. Recall that 
$$\pi_1(M)=\pi_1(M_1)*\dots *\pi_1(M_n)*F_m,$$
where $F_m$ denotes the free group of rank $m$, is taken with respect to a base point in $W$. Let $\delta\in \pi_1(M)$ represent $\alpha$, or more precisely $S^1_\alpha$. If $S$ bounds $M'_i$, then the slide $g_4\colon M\longrightarrow M$ induces an automorphism on $\pi_1(M)$ as given by (\ref{equ: slide M}). 
If $S$ bounds either of the two boundary components of the $j$-th copy of $I \times S^2$, then the slide $g_4\colon M\longrightarrow M$ induces an automorphism as given by (\ref{equ: slide S}).

For more details on slide homeomorphisms, and especially for explicit description of the corresponding Dehn twists, we refer to~\cite[Section 2.2]{Zhao}. 

\medskip

McCullough shows\footnote{As McCullough remarks, his proof was based on an argument by Scharlemann; cf.~\cite[Appendix A]{Bon}.} in \cite[page 69]{Mc} that every orientation-preserving self-homeomorphism of a closed oriented connected $3$-manifold is isotopic to a composite of homeomorphisms of these four types. In fact, McCullough's proof contains even the following more precise form:

\begin{thm} \label{Mc2} 
Let $M$ be a closed oriented connected 3-manifold and $f$ be an orientation-preserving homeomorphism of $M$. Then up to isotopy
\[
f  = g_4\circ g_3\circ g_2\circ g_1,
\]
where each $g_\ell$ is a composition of finitely many homeomorphisms of type $\ell$ on $M$. 
\end{thm}

Observe that if $M$ is a reducible 3-manifold with no aspherical summands in its prime decomposition, i.e. $M=(\#_m S^2\times S^1)\#(\#_{i=1}^n S^3/Q_i)$, where $Q_i$ are finite groups, then $M$ is rationally inessential (i.e. its classifying map is trivial in degree three rational homology) and finitely covered by a connected sum $\# S^2\times S^1$ (this covering corresponds to the kernel of the projection $\pi_1(M)\longrightarrow Q_1\times\cdots\times Q_n$~\cite{KN}). Thus the mapping torus of every homeomorphism $f\colon M\longrightarrow M$ is also rationally inessential and so it has zero simplicial volume. We may therefore assume that the reducible 3-manifold $M$ always contains an aspherical summand in its prime decomposition. Moreover, we may also assume, after passing to finite covering coverings if necessary, that each irreducible summand is aspherical.

By Theorem \ref{Mc2}, the automorphism on $\pi_1(M)$ induced by a self-homeomorphism $f$ of $M$ is a finite composition of the four types of automorphisms of Theorem \ref{t:general}. Thus, in order to show that $\|M\rtimes_f S^1\|=0$, it suffices by Theorem \ref{t:general} to show that the simplicial volume of any mapping torus of each prime summand of $M$ vanishes.

\subsection{Mapping tori of prime 3-manifolds}\label{s:prime}

Now we show that indeed the mapping torus of any prime 3-manifold has zero simplicial volume, and thus finish the proof of Theorem \ref{t:3-mfdoverS^1}.

\subsubsection{Non-aspherical prime fibers}

We begin with the easiest cases, namely, when the 3-manifold fiber $M$ is covered either by $S^2\times S^1$ or by $S^3$. This is actually discussed above, but we include it here as well for completeness.
As for connected sums of such spaces, $M$ is rationally inessential, so is the mapping torus $M\rtimes_f S^1$ for any homeomorphism $f\colon M\longrightarrow M$, and hence $\|M\rtimes_f S^1\|=0$. Alternatively, we can simply invoke the fact that a fiber bundle for which the fiber has amenable fundamental group has vanishing simplicial volume~\cite{Gromov} (which will be used several times below), to conclude that $\|M\rtimes_f S^1\|=0$ for every homeomorphism $f\colon M\longrightarrow M$.

\subsubsection{Irreducible aspherical fibers}

Now, we deal with the remaining cases of mapping tori of prime 3-manifolds, namely with the cases where the fiber is an irreducible aspherical 3-manifold. Recall that a closed irreducible aspherical 3-manifold either possesses one of the geometries $\mathbb{H}^3$, $Sol^3$, $Nil^3$, $\R^3$, $\widetilde{SL_2}$ or $\mathbb{H}^2\times\R$, or it has non-trivial JSJ-decomposition.

\subsubsection*{Hyperbolic fibers}\label{ss:hyperbolic}
If the fiber $M$ is hyperbolic, i.e. it possesses the geometry $\mathbb{H}^3$, then by Mostow rigidity every self-homeomorphism of $M$ is isotopic to a periodic map. This means that 
any mapping torus $M\rtimes_f S^1$ is covered by a product with a circle factor, which has vanishing simplicial volume. By the multiplicativity of the simplicial volume under taking finite coverings, we deduce $\|M\rtimes_f S^1\|=0$.

\subsubsection*{Amenable fibers}
Next, suppose that the fiber $M$ possesses one of the geometries $Sol^3$, $Nil^3$, or $\R^3$. Then $M$ is virtually a mapping torus of $T^2$ \cite{Scott}. In particular, the fundamental group of $M$ fits (up to finite index subgroups) into an extension
\[
1\longrightarrow\Z^2\longrightarrow\pi_1(M)\longrightarrow\Z\longrightarrow 1.
\] 
Since Abelian groups are amenable and extensions of amenable-by-amenable groups are again amenable, we deduce that $\pi_1(M)$ is  amenable. Thus, each mapping torus of $M$ has zero simplicial volume~\cite{Gromov}. 

\subsubsection*{Non-amenable circle bundles}
Now, we deal with the last two geometries, namely with $\widetilde{SL_2}$ and $\mathbb{H}^2\times\R$. If the fiber $M$ possesses one of the latter geometries, then $M$ is virtually a circle bundle over a closed hyperbolic surface~\cite{Scott}. In that case, any mapping torus of $M$ is virtually a circle bundle:

\begin{lem}\label{l:circlebundle}
Let $\Sigma$ be a closed oriented hyperbolic surface and $M$ be a circle bundle over $\Sigma$. Then any mapping torus $M\rtimes_f S^1$ is a circle bundle over a mapping torus of $\Sigma$. 
\end{lem}
\begin{proof}
Let $f\colon M\longrightarrow M$ be a homeomorphism and $\pi\colon M\longrightarrow\Sigma$ be the bundle projection. We observe that $f$ is a bundle map covering a self-homeomorphism of $\Sigma$ \cite{KN,NeoIIPP}: Indeed, since the center $C(\pi_1(M))=\pi_1(S^1)=\Z$ (recall that $M\neq T^3$) and $f_*$ is surjective, we deduce that $(\pi\circ f)_*$ maps the infinite cyclic center of $\pi_1(M)$ to the trivial element in $\pi_1(\Sigma)$. Thus, by the asphericity of our spaces, there is a map $h\colon\Sigma\longrightarrow\Sigma$ such that $h\circ\pi=\pi\circ f$ up to homotopy. Since $f$ has non-zero degree, $h$ must have non-zero degree as well, and by the classification of (maps between) surfaces it must be homotopic to a homeomorphism. We have
\[
M\rtimes_f S^1= (M\times I)/((x,0)\sim(f(x),1))
\]
and
\[
\Sigma\rtimes_h S^1= (\Sigma\times I)/((x,0)\sim(h(x),1)).
\]
Thus, $M\rtimes_f S^1$ can be given an $S^1$-bundle structure with projection map
\begin{align*}
\begin{split}
\overline{\pi}\colon M\rtimes_f S^1 & \longrightarrow  \Sigma\rtimes_h S^1 \\
[(x,t)] & \ \mapsto \ [(\pi(x),t)].
\end{split}
\end{align*}
Indeed, since $h\circ \pi=\pi\circ f$, the map $\overline{\pi}$ is well-defined, namely
\[
\overline{\pi}([(x,0)])=[(\pi(x),0)]=[(h\circ \pi(x),1)]=[(\pi\circ f(x),1)]=\overline{\pi}([(f(x),1)]).
\]
\end{proof}

Since every circle bundle has zero simplicial volume~\cite{Gromov}, Lemma \ref{l:circlebundle} and Lemma \ref{l:finitecover} imply that $\|M\rtimes_f S^1\|=0$ for any homeomorphism $f\colon M\longrightarrow M$.

\subsubsection*{Non-trivial JSJ-decomposition} 

Finally, suppose that our irreducible aspherical 3-manifold $M$ does not possess a Thurston geometry. Then by~\cite{JS,Jo}, there is a non-empty finite collection of disjoint incompressible tori $\mathcal{T}$ such that each component of $M\setminus\mathcal{T}$ is either atoroidal and acylindrical or Seifert fibered; see~\cite{JS,Jo} for explanation of the terminology. If such a collection of tori is minimal, then it is unique up to isotopy, and it is called the JSJ-decomposition of $M$ (after Jacob-Shalen-Johannson~\cite{JS,Jo}). We also refer to the pieces of $M\setminus\mathcal{T}$ as JSJ-pieces. Furthermore, by~\cite{Lue}, we may assume, after possibly passing to a finite cover of $M$, that each JSJ-piece of $M$ is either hyperbolic or fibers over an oriented orbifold of negative Euler characteristic. (As before, we will show that the mapping torus of an iterate of a self-homeomorphism of a finite cover of $M$ has zero simplicial volume; cf.  Lemma \ref{l:finitecover}.)

Suppose now $f\colon M\longrightarrow M$ is an orientation preserving homeomorphism. If $\mathcal{T}$ is a JSJ-decomposition of $M$, then $f(\mathcal{T})$ is also a JSJ-decomposition and so, after isotoping $f$, we can assume that $f(\mathcal{T})=\mathcal{T}$. Thus, after iterating $f$, we may assume that $f$ sends each JSJ-torus and each JSJ-piece to itself. We thus obtain
\[
M\rtimes_f S^1= \mathop{\#_{T^3}}_{i=1 \ \ \ }^{s \ \ \ } M_i\rtimes_{f_i} S^1
\]
where $f_i=f|_{M_i}$, $i=1,...,s$. Since again by Mostow rigidity any homeomorphism of a complete finite volume hyperbolic 3-manifold is isotopic to a periodic map, we can, upon further iteration and isotopy, also suppose that $f$ restricts to the identity on each hyperbolic piece. Thus, if $M_1,...,M_k$ and $M_{k+1},...,M_s$ denote the hyperbolic and Seifert fibered pieces respectively of the JSJ-decomposition, then
by~\cite{BetAl}
\[
\|M\rtimes_f S^1\|\leq \sum_{i=1}^{k}\|M_i\times S^1,T^3\| + \sum_{i=k+1}^s\|M_{i}\rtimes_{f_{i}} S^1,T^3\|,
\]
where the simplicial volumes on the right-hand side of the inequality denote the relative simplicial volumes. Because of the $S^1$ factor, we clearly have $\|M_i\times S^1,T^3\| =0$ for all $i=1,...,k$. Finally, since each $M_i$, $i=k+1,...,s$, fibers over an orbifold with negative Euler characteristic, we deduce, as in Lemma \ref{l:circlebundle}, that each $M_{i}\rtimes_{f_{i}} S^1$ is a circle bundle over a mapping torus of that orbifold, showing therefore that $\|M_{i}\rtimes_{f_{i}} S^1,T^3\|=0$ for all $i=k+1,...,s$. This means that $\|M\rtimes_f S^1\|=0$ as required. 

\medskip

We have finished the proof that every mapping torus of an irreducible aspherical 3-manifold has zero simplicial volume.

\medskip

The proof of Theorem \ref{t:3-mfdoverS^1} is now complete.

\section{Proofs of Corollaries}\label{s:remarks}

\subsection{Fiber bundles in dimension four}

The proof of Corollary \ref{c:fiber4} is a combination of Theorem \ref{t:3-mfdoverS^1} and other known results:

\begin{proof}[Proof of Corollary \ref{c:fiber4}]
Let $M$ be a closed oriented 4-manifold, which is a fiber bundle over a closed manifold $B$, with fiber a closed manifold $F$.

If $\dim(F)=1$, i.e. $F=S^1$, then $\|M\|=0$, because circle bundles have zero simplicial volume by~\cite{Gromov}. 
If $\dim(F)=3$, then $\|M\|=0$ by Theorem \ref{t:3-mfdoverS^1}. Finally, assume that $\dim(F)=\dim(B)=2$. If $F=S^2$ or $B=S^2$, then $M$ is rationally inessential and so $\|M\|=0$ by \cite{Gromov}. If $F=T^2$, then $\|M\|=0$ by the amenability of the fundamental group of the fiber~\cite{Gromov}. If $B=T^2$, then $\|M\|=0$ by~\cite{Bow}. Finally, if both $F$ and $B$ are hyperbolic surfaces, then $\|M\|\geq\|F\times B\|>0$ by~\cite{Buch2} (positivity also follows by the weaker estimate $\|M\|\geq\|F\|\|B\|>0$ of \cite{HK}).
\end{proof}

\subsection{Mapping tori of higher dimensional manifolds}

Theorem \ref{t:3-mfdoverS^1} provides the only dimension (together with dimension two) where all mapping tori have zero simplicial volume:

\begin{proof}[Proof of Corollary \ref{c:mappingtorizero}]
In dimension two, the only closed oriented mapping torus is the the 2-torus which has zero simplicial volume. In dimension four, Theorem \ref{t:3-mfdoverS^1} tells us that every mapping torus has zero simplicial volume. In dimension three, an example of a mapping torus with non-zero simplicial volume is given by a hyperbolic 3-manifold $M$ that fibers over the circle with fiber a hyperbolic surface $\Sigma$. Finally, assume that $N$ is a hyperbolic manifold (or any manifold with $\|N\|>0$) of dimension $n\geq 2$. Then $M\times N$ has dimension $n+3\geq 5$, is a mapping torus of $\Sigma\times N$, and $\|M\times N\|\geq\|M\|\|N\|>0$ by~\cite{Gromov}.
\end{proof}

\subsection{Hyperbolicity of fundamental groups of mapping tori of 3-manifolds}

Theorem \ref{t:3-mfdoverS^1} together with Mineyev's work~\cite{Min1,Min2} implies that the fundamental group of any mapping torus of a rationally essential 3-manifold is never hyperbolic:

\begin{proof}[Proof of Corollary \ref{c:hyperbolic}]
Let $M$ be a closed 3-manifold and $M\rtimes_f S^1$ the mapping torus of a homeomorphism $f\colon M\longrightarrow M$.

If $M$ is irreducible, then by the description and the properties of the 3-manifold groups given in Section \ref{s:prime} it is easy to see that $\pi_1(M\rtimes_f S^1)$ is never hyperbolic, unless $\pi_1(M)$ is finite.

Assume now that $M$ is reducible and contains an aspherical summand in its prime decomposition. (Although not necessary, we can also assume that all aspherical summands are hyperbolic, otherwise $\pi_1(M)$, and hence $\pi_1(M\rtimes_f S^1)$, is not hyperbolic because it has a $\Z^2$-subgroup.) 
Suppose that  $\pi_1(M\rtimes_f S^1)$ is hyperbolic. By the existence of an aspherical summand in (a finite cover of) $M$, we deduce that $M$ and hence also $M\rtimes_f S^1$ are rationally essential. Now Mineyev's work~\cite{Min1,Min2} (see also~\cite{Gromov1}) implies that the comparison map from bounded cohomology to ordinary cohomology
\[
H_b^4(\pi_1(M\rtimes_f S^1);\Q)\longrightarrow H^4(\pi_1(M\rtimes_f S^1);\Q)
\]
is surjective. Thus, by the duality of the $\ell^1$-semi-norm and the bounded cohomology $\ell^\infty$-semi-norm~\cite{Gromov}, we deduce that $\|M\rtimes_f S^1\|>0$. But this contradicts Theorem \ref{t:3-mfdoverS^1}.

Finally, assume that $M$ is reducible and has no aspherical summands in its prime decomposition, i.e. $M=(\#_l S^2\times S^1)\#(\#_{i=1}^n S^3/Q_i)$, where $Q_i$ are finite groups. If $M=\R P^3\#\R P^3$, then $M$ is virtually $S^2\times S^1$ and so $\pi_1(M\rtimes_f S^1)$ is not hyperbolic. If $M\neq \R P^3\#\R P^3$, then $M$ is virtually a connected sum $\#_m S^2\times S^1$, for some $m\geq 2$ (as mentioned before, this covering corresponds to the kernel of the projection $\pi_1(M)\longrightarrow Q_1\times\cdots\times Q_n$). Now, by~\cite{BF1,Br}, $\pi_1(M\rtimes_f S^1)$ is hyperbolic if and only if $\pi_1((\#_mS^2\times S^1)\rtimes S^1)$ does not contain $\Z^2$.
\end{proof}

\subsection{Higher dimensions}

The situation of Theorem \ref{t:general} applies in particular to connected sums of hyperbolic manifolds of dimension greater than two.

\begin{proof}[Proof of Corollary \ref{c:connectedsumshyperbolic}]

Let $M=M_1\#\cdots\# M_n$ be a connected sum of closed oriented hyperbolic manifolds of dimension greater than two, and $f\colon M\longrightarrow M$ be a homeomorphism. Since each $\pi_1(M_i)$ is one-ended, Bass-Serre theory implies that, after possibly passing to a finite iterate of $f$, each $\pi_1(M_i)$ is mapped under $f_*$ to a conjugate of itself. Thus the composition of $f_*$ with a finite number of automorphisms of form (\ref{equ: slide M}) (without $F_m$) has type (1). Since moreover any mapping torus of a hyperbolic closed manifold of dimension greater than two has zero simplicial volume, we deduce by Theorem \ref{t:general} that $\|M\rtimes_f S^1\|=0$.
\end{proof}

\bibliographystyle{amsplain}

\end{document}